\documentclass[preprint,3p,12pt,nonatbib]{elsarticle}

\usepackage[english]{babel}
\usepackage{csquotes}

\usepackage[
    backend=biber,
    style=numeric-comp,  
    sorting=nyt,         
    giveninits=true,     
    maxbibnames=99,      
    doi=true,            
    isbn=false,          
    url=false            
]{biblatex}

\usepackage{xurl}

\DeclareDelimFormat{nametitledelim}{\addcomma\space}

\renewbibmacro{in:}{}

\DeclareNameAlias{author}{given-family}
\DeclareNameAlias{default}{given-family}

\DeclareFieldFormat[article,inproceedings,book,incollection]{title}{#1}
\DeclareFieldFormat{journaltitle}{#1}
\DeclareFieldFormat{booktitle}{#1}

\renewbibmacro*{volume+number+eid}{%
  \printfield{volume}%
  \setunit*{\addspace}%
  \printfield{number}%
  \setunit{\addspace}%
  \printfield{eid}}
\DeclareFieldFormat[article]{number}{\mkbibparens{#1}}

\DeclareFieldFormat{pages}{#1\addperiod}

\DeclareFieldFormat{doi}{%
  \href{https://doi.org/#1}{\nolinkurl{doi:#1}}}

\usepackage{amsmath, amssymb, amsthm, mathtools}

\usepackage{enumitem} 
\usepackage{xcolor}   

\usepackage{setspace}
\usepackage{tikz}
\usetikzlibrary{shapes, arrows, arrows.meta, decorations.pathmorphing, backgrounds, positioning, fit, petri, automata}

\usepackage[
    pdfstartview=XYZ,
    bookmarks=true,
    colorlinks=true,
    linkcolor=blue,
    urlcolor=blue,
    citecolor=blue,
    linktocpage=true,
    hyperindex=true
]{hyperref}

\theoremstyle{plain}
\newtheorem{thm}{Theorem}[section]
\newtheorem{lma}[thm]{Lemma}
\newtheorem{coro}[thm]{Corollary}
\newtheorem{clm}{Claim}

\theoremstyle{definition} 
\newtheorem{rmk}[thm]{Remark}

\numberwithin{equation}{section}

\begin{document}

\begin{frontmatter}

    \title{The matching extendability of optimal 2-planar graphs\tnoteref{t1}}
    \tnotetext[t1]{This work is supported by NSFC (Grant No. 12271229).}

    \author{Xinyao Li}
    \ead{lxinyao2024@lzu.edu.cn}

    \author{Heping Zhang\corref{cor1}}
    \ead{zhanghp@lzu.edu.cn}

    \cortext[cor1]{Corresponding author}
    \address{School of Mathematics and Statistics, Lanzhou University, Lanzhou 730000, P. R. China}

    \begin{abstract}
        A graph is \emph{2-planar} if it can be drawn in the plane such that each edge is crossed by at most two other edges. It is known that for a 2-planar graph $G$, $|E(G)| \le 5|V(G)| - 10$. When the equality holds, we call $G$ an \emph{optimal 2-planar} graph. This paper investigates the matching extendability of optimal 2-planar graphs. By local optimality, we prove that every 4-connected optimal 2-planar graph $G$ of even order is 1-extendable, and give a criterion for $G$ to be 2-extendable. We also prove that no optimal 2-planar graph is 5-extendable and construct a 4-extendable optimal 2-planar graph based on the dodecahedron. Finally, we show that every 6-connected optimal 2-planar graph of even order with at least $2m+2$ vertices is distance 3 $m$-extendable for any $m \ge 0$.
    \end{abstract}

    \begin{keyword}
        $n$-extendable graph \sep 2-planar graph \sep optimal 2-planar graph \sep distance $d$ $m$-extendable graph
    \end{keyword}

\end{frontmatter}

\section{Introduction}\label{sec:introduction}

The matching extendability of graphs embedded on surfaces has been extensively investigated. Plummer~\cite{plummer1988theorem,plummer1992extending}, Fujisawa et al.~\cite{fujisawa2018matching} and Zhang et al.~\cite{ZHANG2023247} respectively obtained some results on the (maximum) matching extendability of planar graphs, optimal 1-planar graphs and general 1-planar graphs. However, any matching extendability of 2-planar graphs has not been revealed yet. This article considers optimal 2-planar graphs. 

The concept of $n$-extendable graphs was originally introduced by Plummer~\cite{plummer1980n}. For a given integer $n \ge 1$, a connected graph $G$ with at least $2n+2$ vertices is said to be \emph{$n$-extendable} if $G$ contains a perfect matching and every matching of size $n$ in $G$ is a subset of a perfect matching.

For planar graphs, Plummer first established the following result.

\begin{thm}[M. D. Plummer~\cite{plummer1988theorem}] \label{thm:planar-ext-bound}
    No planar graph is 3-extendable.
\end{thm}

In 1992, Plummer~\cite{plummer1992extending} further explored 1- and 2-extendable planar graphs.

\begin{thm}[M. D. Plummer~\cite{plummer1992extending}] \label{thm:planar-1ext}
    Every 4-connected planar graph of even order is 1-extendable.
\end{thm}

\begin{thm}[M. D. Plummer~\cite{plummer1992extending}] \label{thm:planar-2ext}
    Every 5-connected planar graph of even order is 2-extendable.
\end{thm}

This line of research has recently been extended to beyond-planar graphs. The \emph{1-planar graphs} form a well-known graph class: in some drawing in the plane, each edge is crossed by at most one other edge. A 1-planar graph has at most $4|V(G)| - 8$ edges, which was established by Pach and Tóth~\cite{Pach1997crossing}. A 1-planar graph is \emph{optimal} if it has exactly $4|V(G)| - 8$ edges. Recently, Fujisawa et al.~\cite{fujisawa2018matching} investigated the matching extendability of optimal 1-planar graphs and obtained the following results.

\begin{thm}[J. Fujisawa, K. Segawa, Y. Suzuki~\cite{fujisawa2018matching}] \label{thm:1-planar-1ext}
    Every optimal 1-planar graph of even order is 1-extendable.
\end{thm}

To characterize the 2-extendability of optimal 1-planar graphs, they introduced the concept of a barrier cycle. A separating cycle $C$ in a 1-planar graph is a \emph{barrier cycle} if every edge of $C$ is uncrossed, and $G-V(C)$ consists of two odd components. Based on this concept, they established the following criteria.

\begin{thm}[J. Fujisawa, K. Segawa, Y. Suzuki~\cite{fujisawa2018matching}] \label{thm:1-planar-2ext}
    Every optimal 1-planar graph of even order is 2-extendable unless it contains a barrier cycle of length 4.
\end{thm}

\begin{thm}[J. Fujisawa, K. Segawa, Y. Suzuki~\cite{fujisawa2018matching}] \label{thm:1-planar-3ext}
    Let $G$ be a 5-connected optimal 1-planar graph of even order and $M$ be a matching of $G$ with $|M| = 3$. Then $M$ is extendable unless $G$ contains a barrier cycle $C$ of length 6 such that $V (M) = V (C)$.
\end{thm}

Additionally, Zhang et al. discovered the following result. 

\begin{thm}[J. Zhang, Y. Wu, H. Zhang~\cite{ZHANG2023247}]\label{thm:o1pg-not-3-ext}
    No optimal 1-planar graph is 3-extendable.
\end{thm}

For general 1-planar graphs, Zhang et al.~\cite{ZHANG2023247} used the discharging method to obtain that no 1-planar graph is 5-extendable and showed that this result is best possible by constructing 4-extendable 1-planar graphs, analogous to Plummer's result for planar graphs (Theorem~\ref{thm:planar-ext-bound}). They further obtained similar results for 1-embedded graphs in other surfaces with small genus~\cite{ZHANG2024114172}. In addition, Huang et al.~\cite{huang2024matching} investigated the matching extendability of 7-connected maximal 1-plane graphs.

Motivated by this research, we investigate the matching extendability of \emph{2-planar graphs}, which can be drawn in the plane such that each edge is crossed at most twice. A 2-planar graph $G$ has at most $5|V(G)| - 10$ edges~\cite{Pach1997crossing}. When $G$ has exactly $5|V(G)| - 10$ edges, we call $G$ an \emph{optimal 2-planar graph}.

Similar to 1-planar graphs, edge contraction in 2-planar graphs does not preserve 2-planarity in general (see~\cite{Chen2005}). In this paper, we overcome this difficulty by generalizing the local optimality to optimal 2-planar graphs. By properly reducing the single-crossing ($K_4$) and double-crossing configurations ($K_5$), we preserve the planarity of the bipartite graph $B(G,S)$ for a vertex subset $S$ (see Section~\ref{sec:construct-bgs}), allowing us to apply Tutte's theorem and Euler's formula to establish the following main theorems.

\begin{thm} \label{thm:o2pg-1ext}
    Every 4-connected optimal 2-planar graph of even order is 1-extendable.
\end{thm}

Furthermore, we show that the connectivity condition in Theorem~\ref{thm:o2pg-1ext} cannot be relaxed by constructing a 3-connected optimal 2-planar graph that is not 1-extendable (see Remark~\ref{rmk:4-conn}).

Next, we investigate the 2-extendability of these graphs and obtain the following result.

\begin{thm} \label{thm:o2pg-2ext}
    Let $G$ be a 4-connected optimal 2-planar graph of even order. Then $G$ is 2-extendable if and only if there is no 4-cycle separating an odd component.
\end{thm}

Theorem~\ref{thm:o2pg-2ext} yields the following corollary.

\begin{coro} \label{coro:o2pg-5conn-2-ext}
    Every 5-connected optimal 2-planar graph of even order is 2-extendable.
\end{coro}

Also, we show that optimal 2-planar graphs are not 5-extendable and this result is tight by constructing a 4-extendable optimal 2-planar graph (see Remark~\ref{rmk:o2pg-not-5-ext}).

\begin{thm} \label{thm:o2pg-not-5ext}
    No optimal 2-planar graph is 5-extendable.
\end{thm}

Aldred and Plummer~\cite{aldred2004edge,ALDRED20102618,aldred2011proximity} introduced distance constraints on matchings extendable to a perfect matching. A connected graph $G$ with at least $2m+2$ vertices is said to be \emph{distance $d$ $m$-extendable} if $G$ contains a perfect matching and any matching $M$ of size $m$, in which any distinct edges are of distance at least $d$, is extendable. For planar triangulations, Aldred and Plummer~\cite{aldred2011proximity} established the following results.

\begin{thm}[Aldred and Plummer~\cite{aldred2011proximity}] \label{thm:aldred-planar}
    If $G$ is a 5-connected even planar triangulation, then
    \begin{itemize}
        \item[(a)] if $m \ge 0$ and $G$ has at least $2m+2$ vertices, then $G$ is distance 5 $m$-extendable and
        \item[(b)] if $0 \le m \le 7$ and $G$ has at least $2m+2$ vertices, then $G$ is distance 4 $m$-extendable.
    \end{itemize}
\end{thm}

\noindent Aldred and Plummer~\cite{aldred2011proximity}, and Aldred and Fujisawa~\cite{aldred2014distance}, respectively, generalized this result to triangulations embedded on the projective plane, the torus and the Klein bottle. For optimal 1-planar graphs $G$ of even order, Theorems~\ref{thm:1-planar-2ext} and~\ref{thm:1-planar-3ext} imply that $G$ is distance 2 2-extendable, and distance 2 3-extendable if $G$ is 5-connected.

In this article we establish the following theorem for optimal 2-planar graphs.

\begin{thm} \label{thm:o2pg-dist3-mext}
    If $G$ is a 6-connected optimal 2-planar graph of even order with at least $2m+2$ vertices, then $G$ is distance 3 $m$-extendable for $m \ge 0$.
\end{thm}

The remainder of this paper is organized as follows. In Section~\ref{sec:preliminaries}, we present preliminary definitions and results. In Section~\ref{sec:construct-bgs}, we introduce locally optimal 2-planar graphs and describe the operations to construct the planar bipartite reduction $B(G,S)$. The proofs of Theorems~\ref{thm:o2pg-1ext} and~\ref{thm:o2pg-2ext}, and~\ref{thm:o2pg-not-5ext} are given in Sections~\ref{sec:proofs-1-2-ext} and~\ref{sec:proof-not-5-ext} respectively. Finally, Section~\ref{sec:proof-distance-ext} is devoted to the proof of the distance-restricted extendability (Theorem~\ref{thm:o2pg-dist3-mext}).

\section{Preliminaries}\label{sec:preliminaries}

In this section, we introduce some basic terminology, notations and related results that will be used throughout the paper. All graphs considered in this paper are finite, simple, and undirected. The \emph{order} of a graph $G$ is the number of its vertices, denoted by $|V(G)|$.

For a vertex $v \in V(G)$, the neighborhood of $v$, denoted by $N_G(v)$, is the set of all vertices adjacent to $v$ in $G$. The degree of $v$, denoted by $d_G(v)$, is the number of edges incident to $v$ in $G$, and the minimum degree of $G$ is denoted by $\delta(G)$. For a subset $U \subseteq V(G)$, the \emph{induced subgraph} $G[U]$ is the graph whose vertex set is $U$ and whose edge set consists of all edges of $G$ that have both endvertices in $U$. For a subset $S \subseteq V(G)$, the graph $G-S$ is the subgraph induced by $V(G) \setminus S$, and we denote by $o(G-S)$ the number of odd components (i.e., components of odd order) in $G-S$. A subset $S \subset V(G)$ is a \emph{vertex cut} if $G-S$ is disconnected. A graph $G$ is \emph{$k$-connected} if $|V(G)| > k$ and $G$ has no vertex cut of size less than $k$. A cycle $C$ in a connected graph $G$ is called a \emph{separating cycle} if $G - V(C)$ contains at least two components. A cycle of length $k$ is referred to as a \emph{$k$-cycle}.

The length of a shortest path between two vertices $u, v \in V(G)$ is called the distance between $u$ and $v$, denoted by $d(u,v)$. We extend this distance definition to edges and vertex sets. For an edge $e = uv$ and a vertex $x$, the distance between them is defined as $d(e, x) = \min \{d(u,x), d(v,x)\}$. For any two edges $e_1, e_2 \in E(G)$, the distance between them, denoted by $d(e_1, e_2)$, is defined as $\min \{d(u,v) \mid u \in V(e_1), v \in V(e_2)\}$.

For an edge $e = xy \in E(G)$, the \emph{edge contraction} of $e$, denoted by $G/xy$, is the operation that removes the edge $xy$, merges the vertices $x$ and $y$ into a single new vertex $v_{xy}$, and connects $v_{xy}$ to every vertex in $(N(x) \cup N(y)) \setminus \{x, y\}$. Finally, any resulting multiple edges are deleted.

Two edges are \emph{independent} if they share no common endvertices. A \emph{matching} $M$ in a graph $G$ is a set of pairwise independent edges. The set of all endvertices of the edges in $M$ is denoted by $V(M)$. A matching $M$ is \emph{perfect} if it covers all vertices of $G$ (i.e., $V(M) = V(G)$). A matching $M$ is \emph{extendable} if $G$ admits a perfect matching containing $M$. The following classical theorem due to Tutte gives a characterization for a graph to have a perfect matching.

\begin{thm}[Tutte's Theorem~\cite{Tutte1947}] \label{thm:tutte}
    A graph $G$ has a perfect matching if and only if $o(G-S) \le |S|$ for every subset $S \subseteq V(G)$.
\end{thm}

If a graph $G$ has no perfect matching, then by Theorem~\ref{thm:tutte}, there exists a subset $S \subseteq V(G)$ such that $o(G-S) > |S|$. Such a subset $S$ is called a \emph{Tutte set} of $G$.

A planar drawing of a graph divides the plane into some regions called \emph{faces}. A graph is \emph{2-planar} if it can be drawn in the plane such that each edge is crossed by at most two other edges. In such a drawing, two edges form a \emph{crossing pair} if they intersect at their interior point, and such edges are referred to as \emph{crossing edges}. An edge is called a \emph{non-crossing edge} if it does not cross any other edge.

It is known~\cite{Pach1997crossing} that a 2-planar graph $G$ has at most $5|V(G)| - 10$ edges. We call $G$ an \emph{optimal 2-planar graph} if $G$ has exactly $5|V(G)| - 10$ edges. The edge density implies that the average degree of $G$ is strictly less than 10 and $\delta(G) \le 9$.

For convenience, we assume a given 2-planar graph $G$ is already drawn in the plane, and we use $G$ to denote both the graph and its drawing. The \emph{true-planar skeleton}, denoted by $\Pi(G)$, is defined as the planar spanning subgraph obtained by removing all crossing edges from $G$.

\begin{lma}[Bekos et al.~\cite{bekos2017optimal}] \label{lma:o2pg-structure}
    Let $G$ be an optimal 2-planar graph. Then $G$ admits a drawing such that its true-planar skeleton $\Pi(G)$ spans all vertices of $G$, contains only faces of length 5, and each face of $\Pi(G)$ has exactly 5 crossing edges in its interior.
\end{lma}

A plane graph in which every face is bounded by a cycle of length 5 is called a \emph{pentangulation}. By Lemma~\ref{lma:o2pg-structure}, the true-planar skeleton $\Pi(G)$ is clearly a pentangulation. Furthermore, Bekos et al.~\cite{BEKOS20191038} established the following connectivity result for this skeleton.

\begin{lma}[Bekos et al.~\cite{BEKOS20191038}] \label{lma:skeleton-3-conn}
    The true-planar skeleton $\Pi(G)$ of an optimal 2-planar graph $G$ is a 3-connected pentangulation.
\end{lma}

\section{Locally Optimal 2-Planar Graphs and Bipartite Reductions}\label{sec:construct-bgs}

\begin{figure}[!b]
    \centering
    \includegraphics[page=1, width=0.6\linewidth]{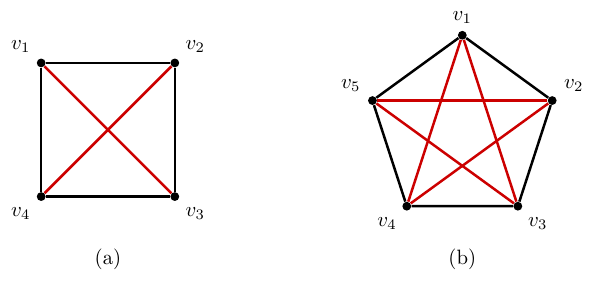}
 \caption{Two crossing configurations in locally optimal 2-planar drawings.}
    \label{fig:crossing-configs}
\end{figure}

In the study of perfect matchings, edge contraction is a commonly used technique. However, contracting an edge in a 2-planar graph may destroy its 2-planarity. To overcome this difficulty, we introduce \emph{locally optimal} 2-planar graphs, similar to the approach used in 1-planar graphs~\cite{fujisawa2018matching}. A 2-planar graph $G$ is said to be \emph{locally optimal} if every crossing edge in $G$ belongs to one of the following two configurations:
\begin{enumerate}[label={\rm(\roman*)}]
    \item \emph{Single-crossing configuration} (or \emph{$K_4$ configuration}): If an edge $v_1v_3$ crosses edge $v_2v_4$, and neither crosses any other edge, then $\{v_1, v_2, v_3, v_4\}$ induces a $K_4$ subgraph, and the edges in this $K_4$ other than $v_1v_3$ and $v_2v_4$ are uncrossed (see Figure~\ref{fig:crossing-configs}(a)).
    \item \emph{Double-crossing configuration} (or \emph{$K_5$ configuration}): If an edge $v_1v_3$ is crossed by two edges $v_2v_4$ and $v_2'v_5$, then $v_2v_4$ and $v_2'v_5$ share exactly one endvertex (say, $v_2=v_2'$), and $\{v_1, v_2, v_3, v_4, v_5\}$ induces a $K_5$ subgraph. Here, $v_iv_{i+2}$ crosses both $v_{i+1}v_{i+3}$ and $v_{i+1}v_{i+4}$ for each $1\leq i\leq 5$ and the subscripts are taken modulo $5$. The remaining edges of this $K_5$ are uncrossed (see Figure~\ref{fig:crossing-configs}(b)).
\end{enumerate}

From Lemma~\ref{lma:o2pg-structure}, we can see that every optimal 2-planar graph is locally optimal. The following lemma shows that locally optimal 2-planarity is preserved under certain operations.

\begin{lma} \label{lma:locally-optimal-preservation}
    Let $G$ be a locally optimal 2-planar graph. Then each of the following operations on $G$ preserves the locally optimal 2-planarity (or simply local optimality):
    \begin{enumerate}[label={\rm(\alph*)}]
        \item deleting a vertex;
        \item contracting a non-crossing edge;
        \item deleting a pair of adjacent crossing edges in a double-crossing configuration;
        \item deleting a crossing edge in a single-crossing configuration.
    \end{enumerate}
\end{lma}

\begin{proof}
    Let $G'$ be the graph obtained after applying one of these operations to $G$. It suffices to show that any crossing edge in $G'$ belongs to either a single-crossing or a double-crossing configuration.

    \begin{figure}[htbp]
        \centering
        \includegraphics[page=2, width=0.55\linewidth]{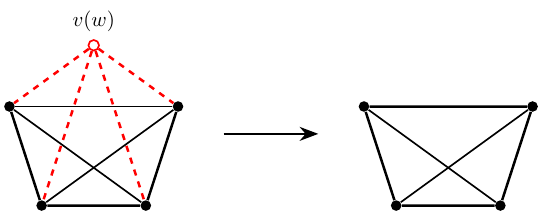}
        \caption{Deletion of vertex $v$ from a double-crossing configuration ($K_5$).}
        \label{fig:vertex-deletion}
    \end{figure}

    (a) For $v \in V(G)$, let $G' = G - v$. Suppose that $e_1$ and $e_2$ cross each other in $G'$. Then $e_1$ and $e_2$ cross each other in $G$. Since $G$ is locally optimal 2-planar, they belong to a single-crossing or a double-crossing configuration $\mathcal{C}$ in $G$. Since $e_1$ and $e_2$ are in $G'$, $v$ is not an endvertex of $e_1$ or $e_2$. If $\mathcal{C}$ is a $K_4$ configuration, $v \notin V(\mathcal{C})$. Thus, the $K_4$ configuration remains in $G'$. If $\mathcal{C}$ is a $K_5$ configuration, then $\mathcal{C}$ contains a fifth vertex $w$. If $v \neq w$, then $v \notin V(\mathcal{C})$, and the $K_5$ configuration remains in $G'$. If $v = w$, then $w$ is deleted (see Figure~\ref{fig:vertex-deletion}). Since the two crossing edges in $\mathcal{C}$ incident to $w$ are removed, $e_1$ and $e_2$ cross in $G'$. The remaining four vertices of $\mathcal{C}$ induce a $K_4$ subgraph in $G'$. The edge in $\mathcal{C}$ that was previously crossed by these two removed edges is no longer crossed and becomes a non-crossing boundary edge. Thus, $e_1$ and $e_2$ are in a single-crossing configuration in $G'$ (see Figure~\ref{fig:vertex-deletion}).

    \begin{figure}[htbp]
        \centering
        \includegraphics[page=3, width=0.9\linewidth]{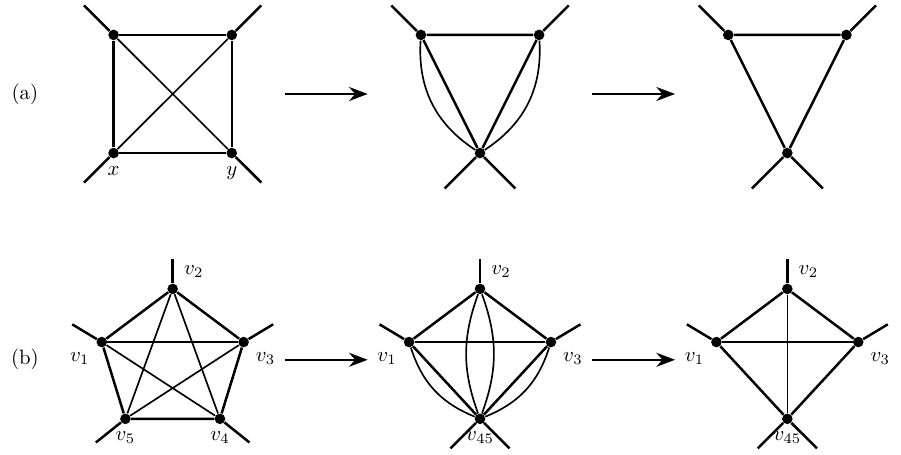}
        \caption{Contraction of a non-crossing edge in single-crossing and double-crossing configurations.}
        \label{fig:contraction}
    \end{figure}

    (b) For a non-crossing edge $xy$, let $G' = G / xy$. Contracting a non-crossing edge does not create new crossings. Thus, any crossing pair $e_1, e_2$ in $G'$ corresponds to a crossing pair in $G$. This crossing pair belongs to a crossing configuration $\mathcal{C}$ in $G$. We consider the intersection $\{x, y\} \cap V(\mathcal{C})$. If $|\{x, y\} \cap V(\mathcal{C})| \le 1$, the contraction of $xy$ does not merge any two vertices of $\mathcal{C}$. Thus, $\mathcal{C}$ remains a valid crossing configuration in $G'$. If $\{x, y\} \subseteq V(\mathcal{C})$, then $xy$ must be a boundary edge of $\mathcal{C}$ since it is uncrossed. If $\mathcal{C}$ is a single-crossing configuration, contracting the boundary edge $xy$ eliminates all crossings in $\mathcal{C}$ and results in a $K_3$ subgraph in $G'$ (see Figure~\ref{fig:contraction}(a)). This contradicts that $e_1$ and $e_2$ are a pair of crossing edges in $G'$. Hence $\mathcal{C}$ is a $K_5$ configuration. Without loss of generality, let $xy$ be the boundary edge $v_4v_5$ (see Figure~\ref{fig:contraction}(b)). Contracting $v_4v_5$ merges $v_4$ and $v_5$ into a single vertex $v_{45}$. The crossing edges $v_2v_4$ and $v_2v_5$ become a single edge $v_2v_{45}$. The double-crossing becomes a single-crossing between $v_1v_3$ and $v_2v_{45}$. In addition, $\{v_1, v_2, v_3, v_{45}\}$ induces a $K_4$ in $G'$. Thus, $e_1$ and $e_2$ are in a single-crossing configuration in $G'$.

    \begin{figure}[htbp]
        \centering
        \includegraphics[page=4, width=0.65\linewidth]{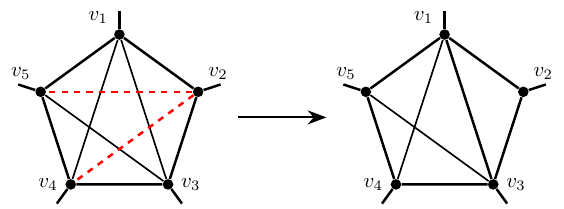}
        \caption{Deletion of crossing edges in a double-crossing configuration.}
        \label{fig:deletion-double}
    \end{figure}

    (c) Let $G'$ be the graph obtained from $G$ by deleting the crossing edges $v_2v_4$ and $v_2v_5$ in a double-crossing configuration $\mathcal{C}^*$ (see Figure~\ref{fig:deletion-double}). In $G'$, the edge $v_1v_3$ is no longer crossed, and thus it becomes a non-crossing edge. The remaining edges of $\mathcal{C}^*$ induce a $K_4$ subgraph on $\{v_1, v_3, v_4, v_5\}$ and a $K_3$ subgraph on $\{v_1, v_2, v_3\}$, sharing the non-crossing edge $v_1v_3$. Since no new crossings are created, any crossing pair $e_1, e_2$ in $G'$ must fall into one of two cases. If the endvertices of $e_1$ and $e_2$ all belong to $V(\mathcal{C}^*)$, they form the unique crossing pair of the single-crossing configuration (i.e., the $K_4$ subgraph on $\{v_1, v_3, v_4, v_5\}$). Otherwise, they must belong to another crossing configuration $\mathcal{C}$ in $G$ ($\mathcal{C} \neq \mathcal{C}^*$). In the latter case, since we only deleted edges within $\mathcal{C}^*$, the crossing configuration $\mathcal{C}$ remains intact in $G'$.

    \begin{figure}[htbp]
        \centering
        \includegraphics[page=5, width=0.65\linewidth]{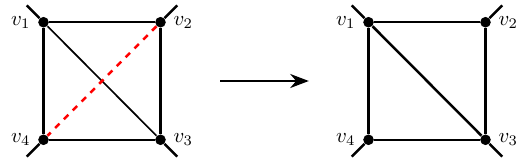}
        \caption{Deletion of a crossing edge in a single-crossing configuration.}
        \label{fig:deletion-single}
    \end{figure}

    (d) Let $G'$ be the graph obtained from $G$ by deleting the crossing edge $v_2v_4$ in a single-crossing configuration $\mathcal{C}^*$ (see Figure~\ref{fig:deletion-single}). In $G'$, the remaining diagonal $v_1v_3$ is no longer crossed, and thus it becomes a non-crossing edge. The remaining edges of $\mathcal{C}^*$ induce two $K_3$ subgraphs on $\{v_1, v_2, v_3\}$ and $\{v_1, v_4, v_3\}$, sharing the non-crossing edge $v_1v_3$. Since no new crossings are created, and the induced subgraph on $V(\mathcal{C}^*)$ in $G'$ contains no crossing edges, any crossing pair $e_1, e_2$ in $G'$ must belong to another crossing configuration $\mathcal{C}$ in $G$ ($\mathcal{C} \neq \mathcal{C}^*$). Since we only deleted an edge within $\mathcal{C}^*$, the crossing configuration $\mathcal{C}$ remains intact in $G'$.
\end{proof}

For a graph $G$ and $S \subseteq V(G)$, let \emph{$B(G,S)$} be the bipartite graph constructed from $G$ as follows: 
\begin{enumerate}[label=(\arabic*)]
    \item remove all even components of $G-S$;
    \item contract each odd component of $G-S$ into a single vertex and delete multiple edges;
    \item delete all edges with both endvertices in $S$.
\end{enumerate}

\begin{lma} \label{lma:bipartite-planar}
    Let $G$ be an optimal 2-planar graph and $S \subseteq V(G)$. Then the bipartite graph $B(G,S)$ is planar.
\end{lma}

\begin{proof}
    We construct $B(G,S)$ from $G$ iteratively. We first remove all even components of $G-S$. By Lemma~\ref{lma:locally-optimal-preservation}(a), this removal preserves the locally optimal 2-planarity. Let $O_1, O_2, \dots, O_k$ be the odd components of $G-S$, for $k\ge 0$. 

    For each odd component $O_i$, we apply contraction and deletion steps until $O_i$ becomes a single vertex $y_i$. We first repeatedly contract all non-crossing edges in $O_i$ and delete multiple edges. By Lemma~\ref{lma:locally-optimal-preservation}(b), these operations preserve the locally optimal 2-planarity. Let $O_i'$ be the resulting graph after these contractions. If $O_i'$ contains edges after this step, they must be crossing edges. Let $e = v_1v_3 \in E(O_i')$ be such an edge. 

    If $v_1v_3$ is crossed by exactly one edge, say $v_2v_4$, then $\{v_1, v_2, v_3, v_4\}$ induces a $K_4$. Since all non-crossing edges in $O_i'$ are contracted, $v_2$ and $v_4$ must belong to $S$. We delete $v_2v_4$. By Lemma~\ref{lma:locally-optimal-preservation}(d), this preserves local optimality. If $v_1v_3$ is crossed by $v_2v_4$ and $v_2v_5$, then $\{v_1, \dots, v_5\}$ induces a $K_5$. Similarly, $v_2, v_4, v_5$ must belong to $S$. We delete $v_2v_4$ and $v_2v_5$. By Lemma~\ref{lma:locally-optimal-preservation}(c), this preserves local optimality.

    These deletions make $v_1v_3$ non-crossing. We then contract it in the next iteration. We repeat this process until each odd component $O_i$ is reduced to a single vertex $y_i$. Let $Y = \{y_1, \dots, y_k\}$. Let $G^*$ be the resulting graph with vertex set $S \cup Y$. By Lemma~\ref{lma:locally-optimal-preservation}, $G^*$ remains a locally optimal 2-planar graph. Note that $Y$ is independent in $G^*$.

    Next, we consider any crossing edge $e$ in $G^*$ that has endvertices separately in $S$ and $Y$. Since $G^*$ is locally optimal, $e$ must belong to a single-crossing or a double-crossing configuration, denoted by $\mathcal{C}$. We claim that $|Y \cap V(\mathcal{C})| = 1$. Otherwise, $|Y \cap V(\mathcal{C})| \ge 2$. Since $\mathcal{C}$ is a complete graph, any two vertices in $Y \cap V(\mathcal{C})$ must be adjacent in $G^*$, contradicting that $Y$ is an independent set. Thus, exactly one vertex of $\mathcal{C}$ belongs to $Y$, and all the other vertices of $\mathcal{C}$ belong to $S$.

    Consequently, the edge (or edges) in $\mathcal{C}$ that crosses $e$ must have both endvertices in $S$. Finally, we construct $B(G,S)$ by deleting all edges with both endvertices in $S$ from $G^*$. This deletion operation naturally removes all edges that cross $e$, rendering every edge between $S$ and $Y$ non-crossing. Therefore, $B(G,S)$ is a planar bipartite graph.
\end{proof}

\section{Proofs of Theorems~\ref{thm:o2pg-1ext} and~\ref{thm:o2pg-2ext}}\label{sec:proofs-1-2-ext}

\begin{proof}[Proof of Theorem~\ref{thm:o2pg-1ext}]
    Let $G$ be a 4-connected optimal 2-planar graph of even order. Suppose to the contrary that $G$ is not 1-extendable. Then there exists an edge $uv \in E(G)$ such that $G-\{u,v\}$ has no perfect matching.

    Let $G' = G-\{u,v\}$. By Tutte's theorem (Theorem~\ref{thm:tutte}), there is a subset $S' \subset V(G')$ such that $|S'| < o(G'-S')$. Let $S = S' \cup \{u,v\}$. Then $G-S=G'-S'$. Since $|V(G')|$ is even, $o(G'-S')$ and $|S'|$ have the same parity. Thus, we have
    $$
    |Y| = o(G-S) = o(G'-S') \ge |S'| + 2 = |S|,
    $$
    where $Y$ is the set of odd components of $G-S$.

    Consider the bipartite graph $B(G,S)$. Since $G$ is 4-connected, each vertex in $Y$ has degree at least $4$ in $B(G,S)$. Thus, we have
    $$
    |E(B(G,S))| \ge 4|Y|.
    $$

    \noindent By Lemma~\ref{lma:bipartite-planar}, $B(G,S)$ is a planar bipartite graph. We have
    $$
    |E(B(G,S))| \le 2|V(B(G,S))| - 4 = 2(|S| + |Y|) - 4.
    $$

    The above two inequalities yield $4|Y| \le 2|S| + 2|Y| - 4$. This implies $|Y| \le |S| - 2$, contradicting that $|Y| \ge |S|$. Therefore, $G$ is 1-extendable.
\end{proof}

\begin{figure}[h]
    \centering
    \includegraphics[page=6, width=0.75\linewidth]{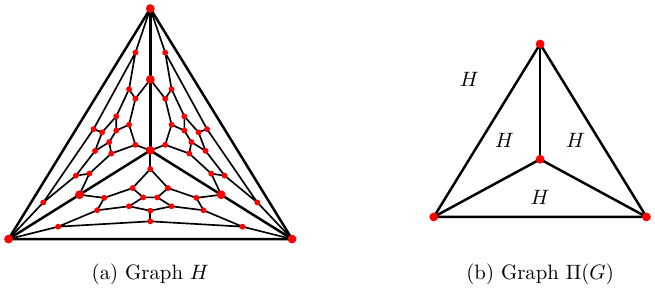}

    \caption{Construction of a 3-connected optimal 2-planar graph that is not 1-extendable. (a) The graph $H$. (b) The 3-connected pentangulation $\Pi(G)$ obtained by embedding $H$ into each triangular face of $K_4$.}\label{fig:3-conn-counter-ex}
\end{figure}

\begin{rmk} \label{rmk:4-conn}
    We note that the connectivity condition in Theorem~\ref{thm:o2pg-1ext} cannot be relaxed to 3-connectivity. We construct a 3-connected optimal 2-planar graph that is not 1-extendable as follows. Let $H$ be the plane graph depicted in Figure~\ref{fig:3-conn-counter-ex}(a) in which the exterior face is bounded by a triangle and all interior faces are pentagons. We embed one copy of the graph $H$ into each triangular face of a complete graph $K_4$ by identifying the 3-cycle of $H$ with the 3-cycle of the copy. This process results in a pentangulation $\Pi(G)$, as illustrated in Figure~\ref{fig:3-conn-counter-ex}(b). By inserting five mutually crossing edges into the interior of each pentagonal face of $\Pi(G)$, we obtain a 2-planar graph $G$. 

    We can verify that $G$ contains a perfect matching. Next, we verify that $G$ is optimal. Since $H$ has 52 vertices, 84 edges, and 33 interior faces, we can compute that $\Pi(G)$ has $200$ vertices, 132 faces and 330 edges. Inserting 5 crossing edges into each face yields $|E(G)| = 330 + 132 \times 5 = 990$. So it is confirmed that $G$ is an optimal 2-planar graph.

    By Lemma~\ref{lma:skeleton-3-conn}, the true-planar skeleton $\Pi(G)$ is 3-connected, so $G$ is 3-connected. However, $G$ is not 4-connected, as the three boundary vertices of each copy of $H$ separate its interior vertices from the rest of the graph, forming a 3-cut.

    We now use Tutte's Theorem to show that $G$ is not 1-extendable. Let $V_0$ be the set of the 4 vertices of the base $K_4$. Removing $V_0$ from $G$ leaves 4 odd components. Let $uv$ be any edge of the base $K_4$, $G' = G - \{u,v\}$ and $S = V_0 \setminus \{u,v\}$. Then $S$ is a Tutte set of $G'$ since $G' - S = G - V_0$ consists of 4 odd components.
\end{rmk}

To prove Theorem~\ref{thm:o2pg-2ext}, we first show some properties of optimal 2-planar graphs.

\begin{lma}\label{lma:edge-pattern-o2pg}
    For each vertex $v$ in an optimal 2-planar graph $G$, the edges incident with $v$ are arranged clockwise in an alternating pattern of one non-crossing edge and two crossing edges. Hence the degree $d_G(v)$ is a multiple of 3.
\end{lma}

\begin{proof}
   By Lemma~\ref{lma:o2pg-structure}, $\Pi(G)$ is a pentangulation, and the five vertices of each pentagonal face in $\Pi(G)$ induce a $K_5$ in $G$. Consider any two consecutive non-crossing edges $e_1$ and $e_2$ incident to $v$ in this clockwise order. These two edges must lie on the boundary of a pentagonal face $f$ in $\Pi(G)$, and there are exactly two crossing edges located between the edges $e_1$ and $e_2$ within $f$. Since this property holds for every pair of consecutive non-crossing edges incident to $v$, the alternating pattern follows, which directly implies that $d_G(v)$ is a multiple of 3.
\end{proof}

\begin{lma}\label{lma:neighbor-connectivity}
    Let $v$ be a vertex in an optimal 2-planar graph $G$. Under a fixed drawing of $G$ as described in Lemma~\ref{lma:o2pg-structure}, let $u_1, u_2, \dots, u_d$ be the neighbors of $v$ arranged clockwise. Then $G$ satisfies the following properties, for $1\le i\le d$ (modulo $d$):

    \begin{enumerate}[label={\rm(\roman*)}]
        \item $u_i u_{i+1}\in E(G)$,  
        \item if $u_i u_{i+2}\notin E(G)$, then $vu_{i+1}$ is a non-crossing edge.
    \end{enumerate}
\end{lma}

\begin{proof}
    (i) Select a non-crossing edge $vu_k$. By Lemma~\ref{lma:edge-pattern-o2pg}, the edges incident to $v$ alternate in one non-crossing edge and two crossing edges. Thus, we have four consecutive neighbors $\{u_k, u_{k+1}, u_{k+2}, u_{k+3}\}$ such that $vu_k$ and $vu_{k+3}$ are non-crossing edges, and $vu_{k+1}$ and $vu_{k+2}$ are crossed. Note that the vertices $\{v, u_k, u_{k+1}, u_{k+2}, u_{k+3}\}$ lie in the boundary of pentagonal face $f$ in $\Pi(G)$ and induce a $K_5$ in $G$, which immediately implies (i).

    (ii) If $u_iu_{i+2} \notin E(G)$, then the vertices $u_i$ and $u_{i+2}$ cannot belong to the same induced $K_5$ corresponding to a pentagonal face in $\Pi(G)$, and they must belong to two different pentagonal faces incident to $v$. Hence $vu_i$ and $vu_{i+2}$ are crossing edges, and $vu_{i+1}$ must be the common boundary shared by these two pentagonal faces in $\Pi(G)$. Since all edges in $\Pi(G)$ are non-crossing, $vu_{i+1}$ is a non-crossing edge.
\end{proof}

\begin{lma} \label{lma:4cut-cycle}
    Let $G$ be a 4-connected optimal 2-planar graph. If $S$ is a 4-cut of $G$, then the induced subgraph $G[S]$ contains a 4-cycle.
\end{lma}

\begin{proof}
    Let $S$ be a 4-cut of $G$. Since $G$ is 4-connected, $S$ is a minimum vertex cut. 
For $v \in S$, $v$ has neighbors in at least two components of $G-S$, say $H_1$ and $H_2$. Let $x_1, x_2, \dots, x_\ell$ be the neighbors of $v$ in $G$ in clockwise order. By Lemma~\ref{lma:neighbor-connectivity}(i), $x_i x_{i+1} \in E(G)$ for all $i$.
If $x_i \in V(H_1)$ and $x_{i+1} \in V(H_2)$, then $x_i x_{i+1}$ connects $H_1$ and $H_2$, contradicting that $H_1$ and $H_2$ are distinct components of $G-S$. Thus, the neighbors of $v$ in $H_1$ and $H_2$ must be separated by vertices in $S$. Therefore, $v$ has at least two neighbors in $S$. Since $|S|=4$, by Ore's Theorem, $G[S]$ contains a 4-cycle.
\end{proof}

\begin{proof}[Proof of Theorem~\ref{thm:o2pg-2ext}]
    We first prove necessity. Suppose to the contrary that $G$ has a 4-cycle $C$ that separates an odd component. Then $C$ has a perfect matching of size 2, which cannot be extended to a perfect matching of $G$. Therefore, $G$ is not 2-extendable, a contradiction.

    Conversely, suppose to the contrary that $G$ is not 2-extendable. Thus, there is a matching $M = \{e_1, e_2\}$ that is not extendable, which means that $G' = G - V(M)$ has no perfect matching. By Theorem~\ref{thm:tutte}, there exists a subset $S' \subseteq V(G')$ such that $o(G' - S') > |S'|$. Let $S = S' \cup V(M)$ and let $Y$ be the set of odd components of $G - S$. Note that $G-S = G'-S'$. We then have $|Y| = o(G' - S') > |S'| = |S| - 4$. Since $|V(G')|$ is even, $o(G' - S')$ and $|S'|$ have the same parity, which forces $|Y| \ge |S| - 2$.

    Consider the bipartite graph $B(G,S)$. Since $G$ is 4-connected, for $B(G,S)$, each vertex in $Y$ is adjacent to at least 4 vertices in $S$. Thus we have
    \begin{equation} \label{eqn:thm-2ext-EBGS-lower}
        |E(B(G,S))| \ge 4|Y|.
    \end{equation}
    
    By Lemma~\ref{lma:bipartite-planar}, $B(G,S)$ is a planar bipartite graph. We have
    $$
    |E(B(G,S))| \le 2|V(B(G,S))| - 4 = 2(|S|+|Y|) - 4.
    $$
    
    The above two inequalities yield $|Y| \le |S| - 2$. Hence $|Y| = |S| - 2$, and $|E(B(G, S))| = 4|Y|$ from Ineq.~\eqref{eqn:thm-2ext-EBGS-lower}. The latter implies that every vertex $y_k \in Y$ has degree exactly 4 in $B(G,S)$. Furthermore, since $V(M) \subseteq S$, we have $|S| \ge 4$, which implies $|Y| = |S| - 2 \ge 2$. So any odd component in $G-S$ has exactly 4 neighbors in $S$, which form a 4-cut separating an odd component. By Lemma~\ref{lma:4cut-cycle}, the induced subgraph of this 4-cut contains a 4-cycle, which separates an odd component, a contradiction. This completes the proof.
\end{proof}

\begin{figure}[htbp]
    \centering
    \includegraphics[page=7, width=0.80\linewidth]{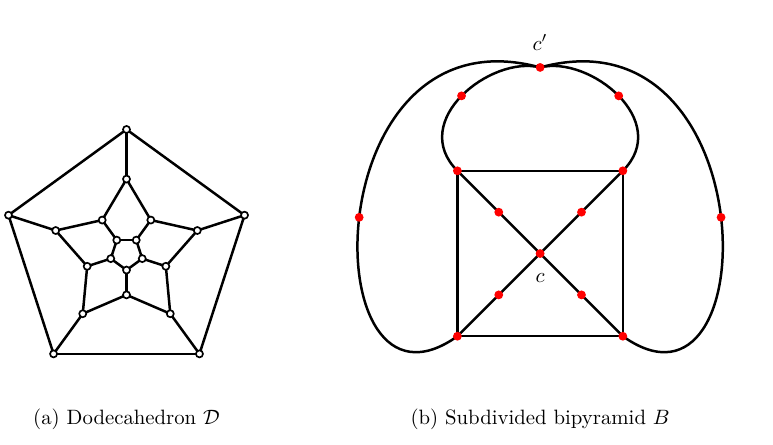}
    \caption{Construction of a non-2-extendable 4-connected optimal 2-planar graph from $\mathcal{D}$ and $B$.}
    \label{fig:4-conn-counter-ex}
\end{figure}

\begin{rmk}
    Theorem~\ref{thm:o2pg-2ext} characterizes the 2-extendability of 4-connected optimal 2-planar graphs by the absence of a 4-cycle that separates an odd component. Here we show that a non-2-extendable 4-connected optimal 2-planar graph exists. 
    
    We first give a subdivided \emph{bipyramid} $B$ as the base graph of our construction as shown in Figure~\ref{fig:4-conn-counter-ex}(b), which has 14 vertices and 8 pentagonal faces. We embed one copy of the dodecahedron graph $\mathcal{D}$ (Figure~\ref{fig:4-conn-counter-ex}(a)) into each pentagonal face of $B$ by identifying a 5-cycle of $\mathcal{D}$ with the 5-cycle of the face. This process results in a pentangulation $\Pi(G)$. By inserting five mutually crossing edges into the interior of each pentagonal face of $\Pi(G)$, we obtain a 2-planar graph $G$.

    We verify that $G$ is optimal. Since $\mathcal{D}$ has 20 vertices, 30 edges, and 11 interior faces, we can compute that $\Pi(G)$ has $134$ vertices, $88$ faces and $220$ edges. Inserting 5 crossing edges into each face yields $|E(G)| = 220 + 88 \times 5 = 660$. So it is confirmed that $G$ is an optimal 2-planar graph.

    Since $G$ is 4-connected, $G$ is 1-extendable by Theorem~\ref{thm:o2pg-1ext}. However, $G$ is neither 5-connected nor 2-extendable by Theorem~\ref{thm:o2pg-2ext} as a 4-cycle (the base-cycle of $B$) separates two odd components of order 65.
\end{rmk}

\section{Proof of Theorem~\ref{thm:o2pg-not-5ext}}\label{sec:proof-not-5-ext}

We first introduce Dean's lemma which provides an important property of $n$-extendable graphs.

\begin{lma}[Dean~\cite{dean1992matching}] \label{lma:dean}
    Let $v$ be a vertex of degree $n + t$ in an $n$-extendable graph $G$. Then $G[N(v)]$ does not contain a matching of size $t$.
\end{lma}

\begin{proof}[Proof of Theorem~\ref{thm:o2pg-not-5ext}]
    Suppose that there exists a 5-extendable optimal 2-planar graph $G$. Since every $k$-extendable graph is $(k+1)$-connected~\cite{plummer1980n}, $G$ is 6-connected, which implies $\delta(G) \ge 6$. On the other hand, since $|E(G)|=5|V(G)|-10$, we have $\delta(G) \le 9$. Let $v$ be a vertex of minimum degree in $G$. Let $\Pi(G)$ be the plane skeleton of $G$ and $d_{\Pi}(v)$ the degree of $v$ in $\Pi(G)$. By Lemma~\ref{lma:edge-pattern-o2pg}, $d_G(v)$ is a multiple of 3 and $d_G(v)=3d_{\Pi}(v)$. Then $d_G(v) \in \{6, 9\}$.

    If $d_G(v) = 6$, then $d_{\Pi}(v) = 2$. However, $\Pi(G)$ is 3-connected by Lemma~\ref{lma:skeleton-3-conn}, a contradiction.

    \begin{figure}[htbp]
        \centering
        \includegraphics[page=8, width=0.40\linewidth]{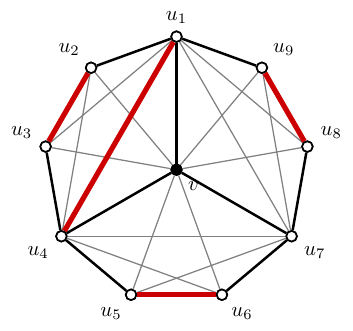}
        \caption{The induced subgraph $G[N(v)]$ for a vertex $v$ with degree 9. The bold edges form a matching of size 4.}
        \label{fig:case2-matching}
    \end{figure}

    If $d_G(v) = 9$, then $d_{\Pi}(v) = 3$. Since $G$ is 5-extendable, by Lemma~\ref{lma:dean}, the induced subgraph $G[N(v)]$ does not contain a matching of size 4. However, $v$ is incident to exactly three pentagonal faces in $\Pi(G)$. In each such face, the five boundary vertices induce a double-crossing configuration in $G$. As illustrated in Figure~\ref{fig:case2-matching}, this induced subgraph contains a matching of size 4 formed by three skeleton boundary edges and one crossing edge, a contradiction. Therefore, no optimal 2-planar graph is 5-extendable.
\end{proof}

\begin{figure}[htbp]
    \centering
    \includegraphics[page=9, width=0.50\linewidth]{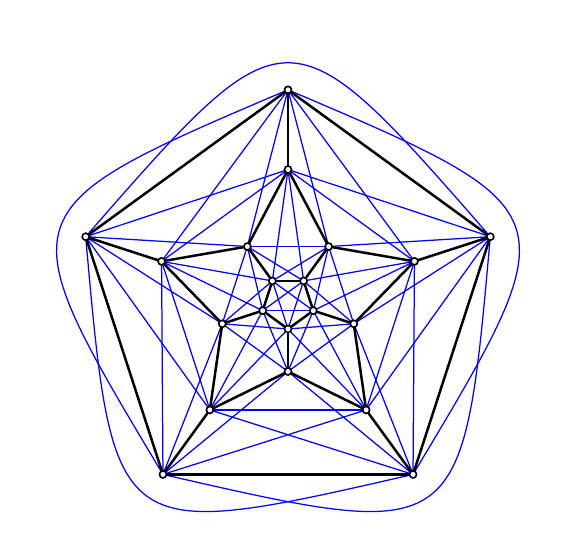}
    \caption{A 4-extendable optimal 2-planar graph whose planar skeleton is the dodecahedron.}
    \label{fig:dodecahedron}
\end{figure}

\begin{rmk}\label{rmk:o2pg-not-5-ext}
    We note that there exists a 4-extendable optimal 2-planar graph. Let $G$ be a 2-planar graph obtained from the dodecahedron graph by inserting 5 mutually crossing edges into each face (see Figure~\ref{fig:dodecahedron}). Then $G$ is an optimal 2-planar graph since $G$ has 20 vertices and $30 + 12 \times 5 = 90$ edges, satisfying $|E(G)| = 5|V(G)|-10$. Furthermore, it is verified by a computer program that $G$ is 4-extendable. 
\end{rmk}

\section{Proof of Theorem~\ref{thm:o2pg-dist3-mext}}\label{sec:proof-distance-ext}

In this section, we present the proof for Theorem~\ref{thm:o2pg-dist3-mext}, establishing the extendability of 6-connected optimal 2-planar graphs under distance constraint.

\begin{proof}[Proof of Theorem~\ref{thm:o2pg-dist3-mext}]
    Throughout the proof, the neighbors of any vertex are always considered clockwise according to a fixed drawing of $G$ as described in Lemma~\ref{lma:o2pg-structure}. Since $G$ is 5-connected, by Corollary~\ref{coro:o2pg-5conn-2-ext} $G$ is 2-extendable, and hence the theorem holds for $m \le 2$. Suppose to the contrary that $m \ge 3$ is the minimum integer such that there exists a 6-connected optimal 2-planar graph $G$ of even order which is not distance 3 $m$-extendable. 

    Let $M = \{e_1, \dots, e_m\}$ be a matching of $G$ such that $d(e_i, e_j) \ge 3$ for $1 \le i < j \le m$, and $M$ is not extendable. Let $G' \coloneqq G - V(M)$. Since $G'$ has no perfect matching, there exists a Tutte set, say $S$, by Tutte's Theorem (Theorem~\ref{thm:tutte}). 

    \begin{clm} \label{clm:tutte-set-equality}
        $o(G' - S) = |S| + 2$.
    \end{clm}

    \begin{proof}[Proof]
        Since $G$ is of even order, $|V(G')| = |V(G)| - 2m$ is even. Hence $o(G' - S) \equiv |S| \pmod 2$. Since $S$ is a Tutte set, we have $o(G' - S) \ge |S| + 2$. Take an edge $uv \in M$. Let $M' = M \setminus \{uv\}$. By the minimality of $m$, the graph $G'' = G - V(M')$ has a perfect matching. Let $S^* \coloneqq S \cup \{u, v\}$. Then $|S^*| = |S| + 2$ and $G'' - S^* = G' - S$. Therefore, $o(G' - S) = o(G'' - S^*) \le |S^*| = |S| + 2$. Hence $o(G' - S) = |S| + 2$.
    \end{proof}

    Let $K \coloneqq S \cup V(M)$. Then $G - K = G' - S$. Let $Y$ be the set of the odd components of $G - K$. By Claim~\ref{clm:tutte-set-equality}, $|Y| = |S| + 2$. We now consider the bipartite graph $B(G, K)$. Since $G$ is 6-connected, each vertex in $Y$ is adjacent to at least 6 vertices in $K$. Thus we have
    \begin{equation}\label{eqn:lower-bound-EBGK}
        |E(B(G,K))| \ge 6|Y| = 6(|S| + 2) = 6|S| + 12.
    \end{equation}

    Moreover, since $B(G, K)$ is a planar bipartite graph by Lemma~\ref{lma:bipartite-planar}, we obtain
    \begin{equation}\label{eqn:upper-bound-EBGK}
    \begin{split}
        |E(B(G,K))| &\le 2|V(B(G,K))| - 4 = 2(|K| + |Y|) - 4 \\
        &= 2(|S| + 2m + |S| + 2) - 4 = 4|S| + 4m.
    \end{split}
    \end{equation}

    Ineqs.~\eqref{eqn:lower-bound-EBGK} and~\eqref{eqn:upper-bound-EBGK} yield
    \begin{equation}\label{eqn:upper-bound-S}
        |S| \le 2m - 6.
    \end{equation}

    Next we will establish a lower bound for $|S|$ to produce a contradiction. For an edge $e = uv$, the neighborhood of $e$ in $G$, denoted by $N_G(e)$, is defined as the set of vertices adjacent to at least one endvertex of $e$, excluding $u$ and $v$ themselves, i.e., $N_G(e) = (N_G(u) \cup N_G(v)) \setminus \{u, v\}$. Accordingly, we define its neighbors in $S$ as $N_S(e) = N_G(e) \cap S$. We then establish the connection between $M$ and the odd components of $G - K$.

    \begin{clm} \label{clm:ei-neighbors}
        Each edge in $M$ has neighbors in at least two distinct odd components of $G - K$.
    \end{clm}

    \begin{figure}[htbp]
    \centering
    \includegraphics[page=10, width=0.7\linewidth]{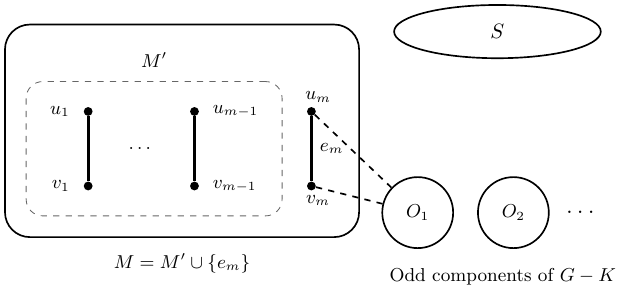}
    \caption{An illustration for parts of $G$: the matching $M = M' \cup \{e_m\}$, a Tutte set $S$ of $G'$ and odd components $O_i$ of $G'-S$.}
    \label{fig:matching-odd-component-connection}
\end{figure}

\begin{proof}
    Suppose for contradiction that an edge in $M$, say $e_m = u_mv_m$, has neighbors in at most one odd component of $G - K$. Let $M' = M \setminus \{e_m\}$ and $G'' = G - V(M')$. By the minimality of $m$, $M'$ can be extended to a perfect matching of $G$. That is, $G''$ has a perfect matching. Hence $o(G''-S)\leq |S|$ by Tutte's theorem. Note that $G' = G'' - u_m - v_m$. Since $e_m$ has a neighbor in at most one odd component of $G' - S$ (see Figure~\ref{fig:matching-odd-component-connection}), $o(G' - S) = o(G'' - S)$. So $o(G'' - S)=o(G' - S) = |S| + 2$ from Claim~\ref{clm:tutte-set-equality}, a contradiction. Hence Claim~\ref{clm:ei-neighbors} is proved.
\end{proof}

    For any edge $e_i = u_i v_i \in M$, we have the following two claims. In the following discussion, we keep in mind that the neighbors of $u_i$ can only lie in $S$ and the components of $G-K$ except $v_i$, since any two edges in $M$ are at distance at least 2 (in fact, at least 3).

    \begin{clm} \label{clm:l-comp-l-1-s}
        If $u_i$ has neighbors in $\ell$ distinct components of $G-K$, then $u_i$ is adjacent to at least $\ell - 1$ vertices in $S$, i.e., $|N_G(u_i) \cap S| \ge \ell - 1$.
    \end{clm} 

    \begin{figure}[htbp]
        \centering
        \includegraphics[page=11, width=0.40\linewidth]{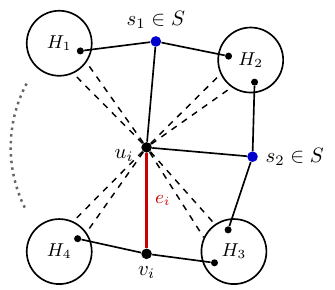}
        \caption{The situation for the neighbors of $u_i$ in $S\cup \{v_i\}$ and the components $H_i$ of $G-K$.}
        \label{fig:ui-s-neighbors}
    \end{figure}

    \begin{proof}
        Lemma~\ref{lma:neighbor-connectivity} establishes that any two consecutive neighbors of $u_i$ are adjacent in $G$. Since there are no edges between distinct components in $G - K$, the neighbors of $u_i$ in distinct components cannot be consecutive. Thus, there are vertices of $K$ between such neighbors, and such vertices must belong to $S \cup \{v_i\}$. Consequently, $u_i$ must be adjacent to at least $\ell - 1$ vertices in $S$ (see Figure~\ref{fig:ui-s-neighbors}).
    \end{proof}

    \begin{clm}\label{clm:ei-s-neighbors}
        $|N_S(e_i)|\geq 2.$
    \end{clm}

    \begin{proof}[Proof]
        Assume for contradiction that $|N_S(e_i)| < 2$.

        If $|N_S(e_i)| = 0$, then neither $u_i$ nor $v_i$ has neighbors in $S$. By Claim~\ref{clm:l-comp-l-1-s}, $u_i$ and $v_i$ each has a neighbor in at most one component of $G-K$. To satisfy Claim~\ref{clm:ei-neighbors}, each of $u_i$ and $v_i$ must connect exactly one component, and these two components are distinct. However, by Lemma~\ref{lma:neighbor-connectivity}, the consecutive neighbors around $u_i$ are adjacent, so $v_i$ is connected to the component that is connected to $u_i$, a contradiction.

        Hence $|N_S(e_i)| = 1$.
        Let $N_S(e_i) = \{s\}$. Without loss of generality, assume that $u_i$ has a neighbor $s$ in $S$. Then $u_i$ has neighbors in at most two components by Claim~\ref{clm:l-comp-l-1-s}.
        
        We then show that $v_i$ is adjacent to $s$. Otherwise, $v_i s \notin E(G)$, which implies that $v_i$ has no neighbors in $S$. Then $v_i$ has a neighbor in at most one component by Claim~\ref{clm:l-comp-l-1-s} again, implying that $u_i$ is connected to at least one component from Claim~\ref{clm:ei-neighbors}. So we consider the following two cases.
        
        If $u_i$ has neighbors in two distinct components, then the two neighbors immediately preceding and succeeding $v_i$ belong to distinct components. It follows from Lemma~\ref{lma:neighbor-connectivity} that these two neighbors are adjacent to $v_i$, implying that $v_i$ has neighbors in two distinct components, a contradiction. If $u_i$ has a neighbor in exactly one component, then $v_i$ must connect the component since the consecutive neighbors around $u_i$ are adjacent by Lemma~\ref{lma:neighbor-connectivity}. That is, $u_i$ and $v_i$ connect only the same component, which is a contradiction to Claim~\ref{clm:ei-neighbors}. Hence, $v_i$ must be adjacent to $s$.

        Then, as with $u_i$, $v_i$ has neighbors in at most two components of $G-K$ by Claim~\ref{clm:l-comp-l-1-s}. We can have that one of $u_i$ and $v_i$ has neighbors in exactly two components of $G-K$. Otherwise, each of $u_i$ and $v_i$ has a neighbor in exactly one component. In this case $u_i$ and $v_i$ are connected to the same component by edges, contradicting Claim~\ref{clm:ei-neighbors}. Without loss of generality, suppose that $u_i$ has neighbors in exactly two components of $G-K$, say $H_1$ and $H_2$. By Lemma~\ref{lma:neighbor-connectivity}, the consecutive neighbors around $u_i$ are adjacent. Further, components $H_1$ and $H_2$ cannot be connected by an edge, so $H_1$ and $H_2$ are separated by $v_i$ and $s$, which implies that $v_i$ also has neighbors in components $H_1$ and $H_2$ (see Figure~\ref{fig:c1-c2-components}).
 \begin{figure}[htbp]
        \centering
        \includegraphics[page=12, width=0.6\linewidth]{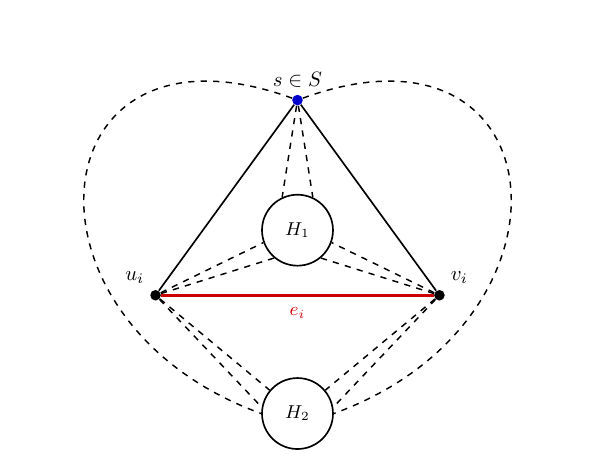}
        \caption{The situation for both $u_is$ and $v_is\in E(G)$, and $e_i = u_iv_i$ has neighbors in exactly two distinct components $H_1$ and $H_2$ of $G-K$.}
        \label{fig:c1-c2-components}
        \end{figure}

        Now we consider the neighbors of $u_i$. The neighbors of $u_i$ in $H_1$ and $H_2$ must be separated by $v_i$ and $s$. Let $w_1 \in V(H_1)$ and $w_2 \in V(H_2)$ be the two neighbors immediately preceding and succeeding $v_i$. Since $H_1$ and $H_2$ are distinct components in $G-K$, $w_1 w_2 \notin E(G)$. By Lemma~\ref{lma:neighbor-connectivity}(ii), $u_i v_i$ is a non-crossing edge. Similarly, $u_i s$ is a non-crossing edge. We apply this argument to the neighbors of $v_i$ and obtain that $v_i s$ is a non-crossing edge (see Figure~\ref{fig:c1-c2-components}). 

        Consequently, $\{u_i, v_i, s\}$ induces a 3-cycle formed entirely by non-crossing edges, which is thus a separating 3-cycle in the planar skeleton $\Pi(G)$. Thus, such a separating 3-cycle remains in $G$, which contradicts the 6-connectivity of $G$. Hence Claim~\ref{clm:ei-s-neighbors} is proved.
    \end{proof}

    For two distinct edges $e_i, e_j \in M$, we claim that $N_S(e_i) \cap N_S(e_j) = \emptyset$. If there exists a vertex $x \in N_S(e_i) \cap N_S(e_j)$, then $d(e_i, x) \le 1$ and $d(e_j, x) \le 1$. This implies $d(e_i, e_j) \le d(e_i, x) + d(x, e_j) \le 1 + 1 = 2$, which contradicts the distance constraint $d(e_i, e_j) \ge 3$. 

    Hence the $m$ subsets $N_S(e_1), \dots, N_S(e_m)$ are pairwise disjoint. By Claim~\ref{clm:ei-s-neighbors}, we obtain
    $$
    |S| \ge \sum_{i=1}^{m} |N_S(e_i)| \ge 2m,
    $$
    which contradicts $|S| \le 2m - 6$ in Ineq.~\eqref{eqn:upper-bound-S}. This completes the entire proof of Theorem~\ref{thm:o2pg-dist3-mext}.
\end{proof}

\printbibliography

@article{fujisawa2018matching,
  author = {Fujisawa, Jun and Segawa, Keita and Suzuki, Yusuke},
  title = {The matching extendability of optimal 1-planar graphs},
  journal = {Graphs and Combinatorics},
  volume = {34},
  number = {5},
  pages = {1089--1099},
  year = {2018},
  doi = {10.1007/s00373-018-1932-6}
}

@inproceedings{bekos2017optimal,
  author = {Bekos, Michael A. and Kaufmann, Michael and Raftopoulou, Chrysanthi N.},
  title = {On optimal 2- and 3-planar graphs},
  booktitle = {33rd International Symposium on Computational Geometry (SoCG 2017)},
  editor = {Aronov, Boris and Katz, Matthew J.},
  series = {Leibniz International Proceedings in Informatics (LIPIcs)},
  volume = {77},
  pages = {16:1--16:16},
  publisher = {Schloss Dagstuhl -- Leibniz-Zentrum f{\"u}r Informatik},
  address = {Dagstuhl, Germany},
  year = {2017},
  doi = {10.4230/LIPIcs.SoCG.2017.16}
}

@article{plummer1980n,
  author = {Plummer, Michael D.},
  title = {On $n$-extendable graphs},
  journal = {Discrete Mathematics},
  volume = {31},
  number = {2},
  pages = {201--210},
  year = {1980},
  doi = {10.1016/0012-365X(80)90037-0}
}

@incollection{plummer1988theorem,
  author = {Plummer, Michael D.},
  title = {A theorem on matchings in the plane},
  booktitle = {Graph Theory in Memory of G. A. Dirac},
  editor = {Andersen, Lars D{\o}vling and Jakobsen, Ivan Tafteberg and Thomassen, Carsten and Toft, Bjarne and Vestergaard, Preben Dahl},
  series = {Annals of Discrete Mathematics},
  volume = {41},
  pages = {347--354},
  publisher = {Elsevier},
  year = {1988},
  doi = {10.1016/S0167-5060(08)70473-4}
}

@article{plummer1992extending,
  author = {Plummer, Michael D.},
  title = {Extending matchings in planar graphs {IV}},
  journal = {Discrete Mathematics},
  volume = {109},
  number = {1},
  pages = {207--219},
  year = {1992},
  doi = {10.1016/0012-365X(92)90292-N}
}

@article{BEKOS20191038,
  author = {Bekos, Michael A. and Di Giacomo, Emilio and Didimo, Walter and Liotta, Giuseppe and Montecchiani, Fabrizio and Raftopoulou, Chrysanthi N.},
  title = {Edge partitions of optimal 2-plane and 3-plane graphs},
  journal = {Discrete Mathematics},
  volume = {342},
  number = {4},
  pages = {1038--1047},
  year = {2019},
  doi = {10.1016/j.disc.2018.12.002}
}

@article{dean1992matching,
  author = {Dean, Nathaniel},
  title = {The matching extendability of surfaces},
  journal = {Journal of Combinatorial Theory, Series B},
  volume = {54},
  number = {1},
  pages = {133--141},
  year = {1992},
  doi = {10.1016/0095-8956(92)90071-5}
}

@article{aldred2011proximity,
  author = {Aldred, Robert E. L. and Plummer, Michael D.},
  title = {Proximity thresholds for matching extension in planar and projective planar triangulations},
  journal = {Journal of Graph Theory},
  volume = {67},
  number = {1},
  pages = {38--46},
  year = {2011},
  doi = {10.1002/jgt.20511}
}

@article{aldred2014distance,
  author = {Aldred, Robert E. L. and Fujisawa, Jun},
  title = {Distance-restricted matching extension in triangulations of the torus and the {Klein} bottle},
  journal = {The Electronic Journal of Combinatorics},
  volume = {21},
  number = {3},
  pages = {\#P3.39},
  year = {2014},
  doi = {10.37236/2952}
}

@article{ZHANG2023247,
  author = {Zhang, Jiangyue and Wu, Yan and Zhang, Heping},
  title = {The maximum matching extendability and factor-criticality of 1-planar graphs},
  journal = {Discrete Applied Mathematics},
  volume = {338},
  pages = {247--254},
  year = {2023},
  doi = {10.1016/j.dam.2023.06.014}
}

@article{Pach1997crossing,
  author = {Pach, J{\'a}nos and T{\'o}th, G{\'e}za},
  title = {Graphs drawn with few crossings per edge},
  journal = {Combinatorica},
  volume = {17},
  number = {3},
  pages = {427--439},
  year = {1997},
  doi = {10.1007/BF01215922}
}

@article{Chen2005,
  author = {Chen, Zhi-Zhong and Kouno, Mitsuharu},
  title = {A linear-time algorithm for 7-coloring 1-plane graphs},
  journal = {Algorithmica},
  volume = {43},
  number = {3},
  pages = {147--177},
  year = {2005},
  doi = {10.1007/s00453-004-1134-x}
}

@article{Tutte1947,
  author = {Tutte, W. T.},
  title = {The factorization of linear graphs},
  journal = {Journal of the London Mathematical Society},
  volume = {s1-22},
  number = {2},
  pages = {107--111},
  year = {1947},
  doi = {10.1112/jlms/s1-22.2.107}
}

@article{ZHANG2024114172,
  author = {Zhang, Jiangyue and Wu, Yan and Zhang, Heping},
  title = {On restricted matching extension of 1-embeddable graphs in surfaces with small genus},
  journal = {Discrete Mathematics},
  volume = {347},
  number = {11},
  pages = {114172},
  year = {2024},
  doi = {10.1016/j.disc.2024.114172}
}

@article{aldred2004edge,
  author = {Aldred, Robert E. L. and Plummer, Michael D.},
  title = {Edge proximity and matching extension in planar triangulations},
  journal = {Australasian Journal of Combinatorics},
  volume = {29},
  pages = {215--224},
  year = {2004}
}

@article{ALDRED20102618,
  author = {Aldred, Robert E. L. and Plummer, Michael D.},
  title = {Distance-restricted matching extension in planar triangulations},
  journal = {Discrete Mathematics},
  volume = {310},
  number = {20},
  pages = {2618--2636},
  year = {2010},
  note = {Graph Theory --- Dedicated to Carsten Thomassen on his 60th Birthday},
  doi = {10.1016/j.disc.2010.03.027}
}

@article{huang2024matching,
  author = {Huang, Yuanqiu and Zhang, Licheng and Wang, Yuxi},
  title = {The matching extendability of 7-connected maximal 1-plane graphs},
  journal = {Discussiones Mathematicae Graph Theory},
  volume = {44},
  number = {2},
  pages = {777--790},
  year = {2024},
  doi = {10.7151/dmgt.2470}
}

\end{document}